\documentclass{amsart}
\usepackage{float}
\usepackage{fourier}
\usepackage{amssymb}
\usepackage{amsmath}
\usepackage{amscd}
\usepackage{amsthm}
\usepackage[centertags]{amsmath}
\usepackage{amsfonts}
\usepackage{newlfont}
\usepackage{graphicx}
\usepackage[usenames]{color}
\usepackage{amsfonts, amssymb}
\usepackage{mathrsfs}
\usepackage{latexsym}
\usepackage{nicefrac}

\usepackage{bbm}
\usepackage[english]{babel} 
\usepackage{verbatim}
\usepackage{tabulary,tabularx}
\usepackage[all]{xy}
\usepackage{pgfplots}
\usetikzlibrary{backgrounds}
\pgfdeclarelayer{background}
\pgfdeclarelayer{foreground}
\pgfsetlayers{background,main,foreground}

\newtheorem{theorem}{Theorem}[section]
\newtheorem{problem}[theorem]{Problem}
\newtheorem{proposition}[theorem]{Proposition}

\newtheorem{lemma}[theorem]{Lemma}

\theoremstyle{definition}

\newtheorem{example}[theorem]{Example}

\theoremstyle{definition}

\theoremstyle{remark}

\newcommand{\R}{\mathbb R}
\newcommand{\Pro}{\mathbb P}
\newcommand{\EE}{\mathbb E}
\newcommand{\vol}{\mbox{vol}}
\newcommand{\M}{\mbox{M}}
\newcommand{\conv}{\mbox{conv}}
\newcommand{\spa}{\mbox{span}}

\newcommand{\bari}{\mbox{bar}}

\begin{document}
\title{The minimal volume of simplices containing a convex body}

\thanks{This work was partially supported by projects CONICET PIP 11220130100329, CONICET PIP 11220090100624, ANPCyT PICT 2015-2299, UBACyT 20020130300057BA. The second author was supported by a CONICET doctoral fellowship.}

\author{Daniel Galicer}
\author{Mariano Merzbacher}
\address{ Departamento de Matem\'{a}tica - IMAS-CONICET,
Facultad de Cs. Exactas y Naturales  Pab. I, Universidad de Buenos Aires
(1428) Buenos Aires, Argentina} \email{dgalicer@dm.uba.ar} \email{mmerzbacher@dm.uba.ar}

\author{Dami\'an Pinasco}
\address{Departamento de Matem\'{a}ticas y Estad\'{\i}sticas, Universidad T. Di Tella, Av. Figueroa Alcorta 7350 (1428), Buenos Aires, Argentina and CONICET}
\email{dpinasco@utdt.edu}



\begin{abstract}
Let $K \subset \R^n$ be a convex body with barycenter at the origin. We show there is a simplex $S \subset K$ having also \emph{barycenter at the origin} such that
$\left(\frac{\vol(S)}{\vol(K)}\right)^{1/n} \geq \frac{c}{\sqrt{n}},$
where $c>0$ is an absolute constant.
This is achieved using stochastic geometric techniques. Precisely, if $K$ is in isotropic position, we present a method to find centered simplices verifying the above bound  that works with extremely high probability.


By duality, given a convex body $K \subset \R^n$ we show there is a simplex $S$ enclosing $K$ \emph{with the same barycenter} such that
\begin{equation*}
\left(\frac{\vol(S)}{\vol(K)}\right)^{1/n} \leq d \sqrt{n},
\end{equation*}
for some  absolute constant $d>0$. Up to the constant, the estimate cannot be lessened.

%

\end{abstract}

\maketitle

%

\section{Introduction}

Approximating a geometric body by a much simpler one results a very common technique in convex geometry and convex analysis, with many applications in discrete geometry and discrete/continuous optimization. For example, the use of the  John/L\"owner ellipsoid (maximum/minimum volume ellipsoid respectively), is one of the most standard tools in these areas \cite{matouvsek2002lectures,gruber2007convex,giannopoulos2001john,lassak1992banach, lassak1998approximation,pelczynski1983structural}. Polytopes, next to ellipsoids, are the most elementary convex sets, chief among them is the simplex. Extremal convex sets for volume ratios of the Euclidean ball  are exactly the simplices \cite{ball1991volume,barthe1998reverse}.

Throughout this article, simplices are  $n$-simplices in $\R^n$ exclusively, i.e. those polytopes formed by the convex hull of $(n+1)$ affinely independent points in $\R^n$ (the vertices).
A convex body in $\R^n$ is a compact convex
set with non-empty interior. For a bounded measurable subset $A \subset \R^n$, we denote by
$\vol(A)$ the volume (or Lebesgue measure) of $A$.

Given a convex body $K \subset \R^n$, we define $$
S(K):=\min{ \left(\frac{\vol(S)}{\vol(K)}\right)^{1/n}},$$
where the minimum is taken over all simplices $S$ in $\R^n$ containing $K$.
An old problem in convex geometry is the following:

\begin{problem}\label{problem0}
How large can $S(K)$ be?
\end{problem}

For the Euclidean plane, i.e. $n=2$, this problem was completely solved by Gross \cite{gross1918affine} (and generalized in different ways by W. Kuperberg \cite{kuperberg1983minimum}): every convex body $K \subset \R^2$ can be inscribed in a triangle of area at most  $2\vol(K)$. This ratio corresponds (exclusively) to the case that $K$ is a parallelogram. The measure of the tetrahedron (not necessarily regular) of least volume circumscribed around a convex body $K \subset \R^3$ is in general unknown. If $K \subset \R^3$ is a parallelepiped of volume one, then the minimal volume tetrahedron containing it has volume $\nicefrac{9}{2}$. It is an open question whether  this is the worst possible fit for the general case. To our knowledge, there are not even conjectured bounds for greater dimensions ($n \geq 4$).

Asymptotic results on this problem  were given in the seventies by Chakerian \cite[Corollary 5]{chakerian1973minimum}. The same estimate was recently rediscovered in 2014 by Kanazawa \cite[Theorem 1]{kanazawa2014minimal} using different arguments. In particular, both authors showed that
\begin{equation}
S(K) \leq n^{\frac{n-1}{n}} \approx  n.
\end{equation}
Note that when $n=2$ this is just Gross' bound.

It is possible to improve the previous bound applying a general inequality for volume ratios due to Giannopoulos and Hartzoulaki \cite{giannopoulos2002volume}. As a consequence of their results we have
\begin{equation}
S(K) \leq  c \sqrt{n} \log(n),
\end{equation}
where $c>0$ is an absolute constant.
Up to our knowledge, this is the best known bound so far; see also the bound given in the recent work of Paouris and Pivovarov \cite[Corollary 5.4]{paouris2017random} on this problem.

One might be interested in requiring additional properties to the simplex. For example, that it shares the same barycenter as the given convex body.
This induces a strong version of the aforementioned problem.
Given a convex body $K \subset \R^n$,  we define $$
S_{\circ}(K):=\min{ \left(\frac{\vol(S)}{\vol(K)}\right)^{1/n}},$$
where the minimum is taken over all simplices $S$ containing $K$ \emph{and having the same barycenter}.
Recall that the barycenter (or center of mass) of a convex body $K$ is given by
\begin{equation}
 \bari(K):= \frac{1}{\vol(K)} \int_K x dx.
\end{equation}

\begin{problem}\label{problem1}
How large can $S_{\circ}(K)$ be?
\end{problem}

Our main result is the following asymptotic estimate on this problem.

\begin{theorem}\label{main theorem}
Let $K \subset \R^n$ be a convex body. There is a simplex $S$ enclosing $K$ \emph{with the same barycenter} such that
\begin{equation} \label{boundsimplex}
\left(\frac{\vol(S)}{\vol(K)}\right)^{1/n} \leq d \sqrt{n},
\end{equation}
for some absolute constant $d>0$.

\end{theorem}


In fact for a centrally symmetric convex body $K$ we prove the following bound:
\begin{equation} \label{boundsimplexconcteisotropia}
S_{\circ}(K) \leq \frac{d \sqrt{n}}{L_{K^{\circ}}},
\end{equation}
for some absolute constant $d>0$. Here $L_{K^{\circ}}$ stands for the isotropic constant of the polar body $K^{\circ}$ (see definitions below). If $K$ is an arbitrary body (not necessarily centrally symmetric), we have 
\begin{equation} \label{boundsimplexconcteisotropianosim}
S_{\circ}(K) \leq \frac{d \sqrt{n}}{L_{D(K)^{\circ}}}, 
\end{equation}
where $D(K)$ stands for the difference body $K - K$.
Note that Equation~\eqref{boundsimplexconcteisotropia} is a direct consequence of  Equation~\eqref{boundsimplexconcteisotropianosim} (if $K$ is centrally symmetric then the difference body $D(K)$ is just  $2\cdot K$).

The estimate \eqref{boundsimplex} above, up to the absolute constant $d>0$, cannot be improved. Indeed, we show in Example~\ref{ejemplo bola} that if $K = B^n_2$, the Euclidean unit ball, then the regular simplex circumscribing it (see Figure~\ref{simplex y bola}) is a minimal volume simplex that contains $K$; and therefore
\begin{equation}
S(K)=S_{\circ}(K) \geq \widetilde{d}\sqrt{n},
\end{equation}
for some positive constant $\widetilde{d}>0$.

%

%


By duality,  Problem~\ref{problem1} is related with finding simplices of large volume \emph{inside} a convex body with the same barycenter.
The search of simplices of large volume contained in a convex body has an extensive and interesting history in geometry. For instance, the study of the maximum area of triangles in planar
convex bodies was undertaken by Blaschke \cite{blaschke1917affine} in the early 20th century.
Sas \cite{sas1939extremumeigenschaft} and Macbeath \cite{macbeath1951extremal} also considered the problem of approximating a given
convex body by inscribed polytopes.
Mckinney \cite{mckinney1974maximal}  studied certain properties of those simplices of maximum volume inside a  centrally symmetric convex body.
The survey \cite{hudelson1996largest} also deals with simplices of large volume in cubes.

Given a convex body $K \subset \R^n$ with barycenter at the origin, we focus on finding a simplex $ S \subset K$ of  large volume having also \emph{barycenter at the origin}.
Our contribution is the following.

\begin{theorem} \label{Theorem A}
Let  $K \subset \R^n$ be a convex body with barycenter at the origin. There is a simplex $S \subset K$ with barycenter at the origin such that
\begin{equation}
\left(\frac{\vol(S)}{\vol(K)}\right)^{1/n} \geq \frac{c}{\sqrt{n}},
\end{equation}
where $c>0$ is an absolute constant.
\end{theorem}

Recall that the Mahler product of a given convex set $K \subset \R^n$ is defined as
\begin{equation}
M(K):= \vol(K) \vol(K^{\circ}),
\end{equation}
where $K^{\circ}$ stands for the polar set of $K$, i.e.
\begin{equation}
K^{\circ}= \{ x \in \R^n : \langle x, y \rangle \leq 1 \mbox{ for all } y\in K \}.
\end{equation}
One of the reasons we restrict ourselves  in searching for simplices  having \emph{barycenter at origin} is because we know the exact value of their Mahler product (see Lemma~\ref{Mahler}).

Our approach to obtain Theorem~\ref{main theorem} is based on a very simple idea. 
Loosely speaking, it is not difficult to see that the problem can be reduced to the case in which $ K $ is centrally symmetric (and therefore $K^\circ$ has barycenter at the origin). Note that, by Theorem~\ref{Theorem A}, there is a simplex $T\subset K^{\circ}$ of \emph{large} volume having also barycenter at the origin. By duality, we have that $K$ is enclosed by the simplex $T^{\circ}$ (with barycenter at the origin), which we show that has small volume. To do this, we make use of its Mahler product (since $T$ is centered) and the reverse Blaschke-Santal\'o inequality  (also known as the  Bourgain-Milman inequality, \cite[Theorem 8.2.2]{artstein2015asymptotic})  for the body $K$.

\bigskip

Stochastic geometry studies randomly generated geometric objects. We use techniques from this area to prove Theorem~\ref{Theorem A}. Indeed, this theorem is a consequence of a more general result of probabilistic nature (see Theorem~\ref{Theorem B} below).

Before we go into more detail we recall some basic definitions and set some notation. We denote the family of all simplices in $\R^n$ with barycenter at the origin  by $\mathcal{S}^n_0$. We write $S^{n-1}$ for the Euclidean sphere in $\R^n$ and denote by $\langle \cdot, \cdot \rangle$ the standard scalar product in $\R^n$.

A convex body is said to be in isotropic position (or simply, is isotropic) if it has volume one and satisfies the following two conditions:
\begin{itemize}
\item $\displaystyle{\int_Kx \, dx=0 \textrm{ (barycenter at 0)},}$
\item $\displaystyle{\int_K\langle x,\theta\rangle^2 \, dx=L_K^2\quad \forall \theta\in S^{n-1},}$
\end{itemize}
where $L_K$ is a constant independent of $\theta$, which is called the isotropic constant of $K$.

It is not hard to see that for every convex body $K$ in $\R^n$ with center of mass at the origin, there exists $A \in GL(n)$ such that $A(K)$ is isotropic \cite[Proposition 10.1.3]{artstein2015asymptotic}. Moreover, this isotropic image is unique up to orthogonal transformations; consequently, the isotropic constant $L_K$ results an invariant of the linear class of $K$. In some sense, the isotropic constant $L_K$ measures the spread of a convex
body $K$.

If the convex body $K$ is in isotropic position, the following theorem gives a probabilistic method to find simplices  inside $K$ (having barycenter at the origin)  with volume large enough. We believe this result is interesting in its own right.

\begin{theorem} \label{Theorem B}
There exists a function $f_n: \underbrace{\R^n \times \dots \times \R^n}_{n}  \to \mathcal{S}^n_0$ such that for every isotropic convex body $K \subset \R^n$ and $X_1, \dots, X_n$ independent random vectors uniformly distributed on $K$, then with probability greater than $1 - e^{-n}$ we have that $f_n(X_1,\dots,X_n)$ is a simplex with barycenter at the origin contained in $K$ such that
\begin{equation} \label{desigualdad adentro proba}
\vol(f_n(X_1,\dots,X_n)) \geq \frac{c^n L_K^{n} }{n^{\nicefrac{n}{2}}},
\end{equation}
where $c>0$ is an absolute constant.
\end{theorem}

Note that the volume ratio is invariant under linear transformations  i.e.,
\begin{equation}\label{invariantetransf}
\left(\frac{\vol(S)}{\vol(K)}\right)^{1/n} = \left(\frac{\vol(A (S) )}{\vol(A(K))}\right)^{1/n}
\end{equation}
for every $A \in GL(n)$. Thus, we have the following result:

{\sl For every convex body $K \subset \R^n$ with barycenter at the origin there is a simplex $S \subset K$ having also barycenter at the origin such that
\begin{equation} \label{desigualdad adentro}
\left(\frac{\vol(S)}{\vol(K)}\right)^{1/n} \geq \frac{c L_K}{\sqrt{n}},
\end{equation}
where $c>0$ is an absolute constant.
}

Observe that, since the isotropic constant $L_K$ of any convex body is bounded from below by and absolute constant \cite[Proposition 10.1.8.]{artstein2015asymptotic}, then  Theorem~\ref{Theorem A} follows from the previous inequality. We emphasize that it is unknown whether the isotropic constant is bounded from above by an absolute constant. The best known general upper bound  $L_K\leq cn^{\frac{1}{4}}$, which was given by  Klartag \cite{klartag2006convex} and improves the earlier estimate $L_K\leq cn^{\frac{1}{4}}\log n$ due to Bourgain \cite{bourgain1991distribution}.

\bigskip

Estimates \eqref{desigualdad adentro proba} and \eqref{desigualdad adentro} should also be contrasted  with a classic result of Macbeath \cite{macbeath1951extremal} (see also \cite[Theorem 2.10.]{pach2011combinatorial}), which asserts that any convex body $K \subset \R^n$ contains a convex polytope of $d$ vertices, whose volume is at least as large as the maximal volume of a polytope of $d$ vertices inscribed in a Euclidean ball (of the same volume as $K$). In particular, if $d=n+1$ it is not difficult to see that we can find a simplex $S \subset K$ such that  
\begin{equation}\label{resultado de Macbeath}
\left(\frac{\vol(S)}{\vol(K)}\right)^{1/n} \geq \frac{1}{\sqrt{n}},
\end{equation}
where $c>0$ is an absolute constant.
The same can be deduced using the well-known Dvoretzky theorem. The inequalities given in \eqref{desigualdad adentro proba}, \eqref{desigualdad adentro} and \eqref{resultado de Macbeath} resemble, at first glance, the asymptotic growth given by Milman and Pajor in \cite[Proposition 5.6.]{milman1989isotropic} (connected with the $n$-dimensional generalization of the classical Sylvester problem, see \cite{brazitikos2014geometry} and the references therein).
On the other hand, Equations \eqref{desigualdad adentro proba} and \eqref{desigualdad adentro} can also be linked with the bounds given when applying the Blaschke-Groemer inequality \cite[Theorem 8.6.3]{schneider2008stochastic} (or the Busemann random simplex inequality, see \cite[Theorem 9.2.6]{gardner1995geometric} or \cite[Theorem 8.6.1]{schneider2008stochastic}), which state that the expected volume of a randomly generated simplex inside a given convex body $K$ of fixed volume is minimized when $K$ is an ellipsoid.
Anyway, either by the results of Milman and Pajor \cite[Proposition 5.6 ]{milman1989isotropic} or by the Blaschke-Groemer inequality \cite[Theorem 8.6.3]{schneider2008stochastic}, for every convex body $K \subset \R^n$ one gets
\begin{equation} \label{cota esperanza}
\mathbb{E}_{X_i \in K} [\vol\left(\conv(X_1, \dots, X_{n+1})\right)] \geq \frac{c^n \vol(K)^n}{n^{\nicefrac{n}{2}}},
\end{equation}
where $c>0$ is an absolute constant.

Our contribution, Theorem \ref{Theorem B}, consists in giving \emph{with extremely high probability}, simplices with the \textbf{same barycenter} (a key property for our purposes) whose volumes satisfy the same lower bound: of order $\frac{\vol(K)^n}{n^{\nicefrac{n}{2}}}$. The main idea to get this is to show that we can find with extremely high probability randomly generated simplices whose  \emph{barycenters are close to the origin} (Proposition \ref{Propo 1}) and \emph{with large volume} (see Proposition \ref{Propo volumen}; this should also be compared with Equation \eqref{cota esperanza} above). Then we make a suitable rescale to make the centroids match, with the care to keep staying within the original body. All this is inspired, in a sense, on some arguments presented on the recent paper of Nasz{\'o}di \cite{naszodi2016proof}, which solves a conjecture of {B}{\'a}r{\'a}ny, {K}atchalski and {P}ach regarding quantitative Helly type results (see also the proof of \cite[Theorem 3.1.]{brazitikos2017brascamp}).

\bigskip
The article is organized as follows. In Section~\ref{Section 2} we give a proof of  Theorem~\ref{Theorem B}. Then, in Section~\ref{Section 3} we prove Theorem~\ref{main theorem}, and show that the corresponding asymptotic  estimate is sharp.
We refer the reader to the books \cite{artstein2015asymptotic} and \cite{brazitikos2014geometry} for the general  theory of asymptotic geometric analysis and the theory of isotropic convex bodies.
In Section~\ref{graciasref} we have included  an alternative proof of  Theorem~\ref{main theorem} based on some enlightening comments given by the anonymous referee.

\section{A probabilistic approach}\label{Section 2}

The probabilistic method is a standard method for proving the existence of a specified kind of mathematical object. The philosophy is to show that if one randomly chooses objects from a specified class, the probability that the result is of the prescribed type is positive.
In this section we use this method to give a proof Theorem~\ref{Theorem B}.
For this we need two propositions that essentially state that, with very high probability, certain random simplices have ``good properties''.

Before we state them, we recall some basic properties on simplices and convex bodies.
We denote by $S(v_0, \dots, v_n)$ the convex hull of the points $v_0, \dots, v_n \in \R^n$ or, in other words, the simplex with vertices $v_0, \dots, v_n$.
It is easy to see that the barycenter/centroid of a simplex $S(v_0, \dots, v_n)$ is given by the mean of the vertices,
\begin{equation}
\bari(S(v_0, \dots, v_n))= \frac{1}{(n+1)} \sum_{i=0}^n v_i.
\end{equation}

%
%
%

Suppose $K \subset \R^n$ is an isotropic convex body and we randomly choose $X_1, \dots, X_n$ in $K$. The following statement asserts that typically  the barycenter of the random simplex $S(0, X_1, \dots, X_n)$ has ``small'' norm.

\begin{proposition} \label{Propo 1}
There is an absolute constant $c_1>0$ such that for every isotropic convex body $K\subset \R^n$ and  $\{X_i\}_{i=1}^n$ independent random vectors uniformly distributed in $K$ then
\begin{equation}
\Pro\left\{  \left\Vert \bari(T) \right\Vert \leq c_1 L_K \right\} > 1 - \frac{1}{2} e^{-n},
\end{equation}
where  $T$ is the random simplex $S(0, X_1, \dots, X_n)$.
\end{proposition}

Our arguments to prove this proposition are based on the proofs of \cite[Theorem 3.1.]{alonso2008isotropy} and \cite[Theorem 1.1.]{klartag2009hyperplane}.
We need to state two lemmas. For elementary background on Orlicz spaces we refer the reader to \cite[Section 3.6.2.]{artstein2015asymptotic}.


The first fact we need, Lemma~\ref{Fact psi1} below, asserts a ``good behavior'' of the marginals $\langle \cdot, \theta \rangle$, for any direction $\theta \in S^{n-1}$.

\begin{lemma} \label{Fact psi1}
There is an absolute constant $C>0$ such that for every isotropic convex body $K \subset \R^n$ and every $\theta \in S^{n-1}$ we have
\begin{equation}
\Vert \langle \cdot, \theta \rangle \Vert_{L_{\psi_1}} \leq C L_K.
\end{equation}
\end{lemma}

The previous statement is known in the area and is a direct consequence of  \cite[Lemma 3.5.5.]{artstein2015asymptotic} and \cite[Theorem 3.5.11]{artstein2015asymptotic}.

We also need a classical inequality due to Bernstein about sums of independent random variables (see, for example \cite[Theorem 3.5.16]{artstein2015asymptotic}).

\begin{theorem}[Bernstein inequality] \label{Bernstein}
Let $\{Y_i\}_{i=1}^n$ be a sequence of random variables with mean $0$ on some probability space. Assume that $Y_i$ belong to $L_{\psi_1}$ and that $\Vert Y_i \Vert_{L_{\psi_1}} \leq M$ for all $i = 1, \dots, n$. Let $\sigma^2 = \frac{1}{n} \sum_{i=1}^n \Vert Y_i \Vert_{L_{\psi_1}}^2$.
Then, for all $t >0$,
\begin{equation}
 \Pro \left\{ \left| \sum_{i=1}^n Y_i \right| > t n \right\} \leq  e^{ -D \; n \min\{\frac{t^2}{\sigma^2},\frac{t}{M}\}},
\end{equation}
for some absolute constant $D>0$.
\end{theorem}
We are now ready to give a proof of Proposition~\ref{Propo 1}.

\begin{proof}[Proof of Proposition~\ref{Propo 1}]
Let $\{X_i\}_{i=1}^n$ be independent random vectors uniformly distributed on $K$ and let $\theta$ be fixed direction in $S^{n-1}$.

By combining Lemma~\ref{Fact psi1} and Theorem~\ref{Bernstein} for the random variables $Y_j:=\langle X_j, \theta \rangle$ we have, for all $t > CL_K$,
$$ \Pro \left\{ \left\vert \langle \sum_{i=1}^n X_i, \theta \rangle \right\vert > t n \right\} \leq e^{-n  \frac{ \; t \; D}{C \; L_K}}.$$

Let $\mathcal N$ be a $\frac{1}{2}$-net on the sphere of cardinality less than or equal to $5^n$ (see e.g., \cite[Lemma 5.2.5.]{artstein2015asymptotic}). Then
$$ \Pro \left\{ \left\vert \langle \sum_{i=1}^n X_i, \theta \rangle \right\vert > t n \;\; \mbox{for some $\theta \in \mathcal N$}\right\} \leq e^{-n(\frac{t \; D}{C \; L_K} - \log(5))},$$ and hence
$$ \Pro \left\{ \left\vert \langle \sum_{i=1}^n X_i, \theta \rangle \right\vert \leq t n \;\; \mbox{for every $\theta \in \mathcal N$}\right\} \geq 1- e^{-n(\frac{t \; D}{C \; L_K} - \log(5))}.$$

Every vector $\vartheta \in S^{n-1}$ can be written in the form $\vartheta = \sum_{j=1} \delta_j \theta_j$, with $\theta_j \in \mathcal N$ and $0 \leq \delta_j \leq 2^{1-j}$ (see for example the proof of \cite[Proposition 5.2.8.]{artstein2015asymptotic}).

Observe that
$$ \bigcap_{\theta \in \mathcal{N}} \left\{\left\vert \langle \sum_{i=1}^n X_i, \theta \rangle \right\vert \leq t n \right\}  \subset \left\{ \left\Vert \sum_{i=1}^n X_i \right\Vert \leq 2 t n \right\}  =\left\{\max_{\vartheta \in S^{n-1}} \left\vert \langle \sum_{i=1}^n X_i, \vartheta \rangle \right\vert \leq 2 t n \right\}.$$
Indeed, let $\vartheta$ be an arbitrary unit vector and suppose that $\vert \langle \sum_{i=1}^n X_i, \theta \rangle \vert \leq t n$ for every $\theta \in \mathcal N$, then
$$ \left\vert \langle \sum_{i=1}^n X_i,  \vartheta \rangle \right\vert = \left\vert \langle \sum_{i=1}^n v_i,  \sum_{j=1}^\infty \delta_j \theta_j \rangle \right\vert \leq \sum_{j=1}^\infty \delta_j \left\vert \langle \sum_{i=1}^n X_i, \theta_j \rangle \right\vert \leq 2 t n.$$
Thus, for every $t > C L_K$ we  have
$$\Pro \left\{ \left\Vert \sum_{i=1}^n X_i \right\Vert \leq 2 t n \right\} \geq 1 - e^{-n(\frac{t \; D}{C \; L_K} - \log(5))}.$$
The result now follows by setting $t := \frac{c_1 (n+1) L_K}{2n}$, for $c_1>0$ sufficiently large.
\end{proof}

The second proposition we need asserts that the simplex $S(0, X_1, \dots, X_n)$ typically  has   ``large volume''.

\begin{proposition} \label{Propo volumen}
There is an absolute constant $c_2>0$ such that for every isotropic convex body $K\subset \R^n$ and  $\{X_i\}_{i=1}^n$ independent random vectors uniformly distributed in $K$ then
\begin{equation}
\Pro\left\{  \vol\left(S(0, X_1 \dots, X_n) \right) \geq \frac{c_2^n L_K^n}{n^{\frac{n}{2}}} \right\} > 1 - \frac{1}{2} e^{-n}.
\end{equation}

\end{proposition}

A proof of it can be found essentially in the work of Pivovarov \cite[Proposition
1]{pivovarov2010determinants}. We include the details for completeness.

\begin{lemma}{\cite[Lemma 2]{pivovarov2010determinants}} \label{Pivo lema2}
Let $K\subset \R^n$ an isotropic convex body and $X$ be a random vector uniformly distributed on $K$. Let $E \subset \R^n$ be a $k$-dimensional subspace and $P_E$ the orthogonal projection onto $E$.
Then the random variable
$$Y:=\frac{\left\vert P_E(X) \right\vert}{L_K \sqrt k}$$
satisfies
$$\mathbb{E} \left\vert Y \right\vert^{-\frac{1}{2}} \leq C',$$
where $C'>0$ is an absolute constant.
\end{lemma}

%
%
%

\begin{proof}[Proof of Proposition~\ref{Propo volumen}]

Let $A: \R^n \to \R^n$ be the linear transformation mapping the canonical basis $\{e_i\}_{i=1}^n$ to $\{X_i\}_{i=1}^n$. We have
\begin{equation}
\vol(S(0,X_1,\dots,X_n)) = \frac{|\det(A)|}{n!}.
\end{equation}

Set $V_k := \spa \{X_1,\dots,X_k\}$ and $Y_k = \frac{|P_{{V_K}^\bot} X_k|}{L_K \sqrt{n-k+1}}$. Note that by Lemma~\ref{Pivo lema2} if $X_1,\dots,X_{k-1}$ are fixed we have $\EE[|Y_k|^{-\frac{1}{2}}] \leq C'$.

Using the fact that
\begin{equation}
|\det(A)| = \Vert X_1 \Vert \Vert  P_{{V_1}^\bot}(X_2) \Vert \dots \Vert P_{{V_{n-1}}^\bot} (X_n)\Vert
\end{equation}
and applying Fubbini theorem iteratively we obtain
\begin{equation} \label{equation esperanza negativa}
\EE[\prod_i^n Y_k^{-\frac{1}{2}}] \leq (C')^n .
\end{equation}
Let $\alpha >0$ be a constant to be determined. Then by Markov inequality and Equation~\eqref{equation esperanza negativa} we have
\begin{align*}
\Pro(|\det(A)| < \alpha^n L_K^n \sqrt{n!}) &= \Pro(\prod_i^n Y_k < \alpha^n)\\
& = \Pro(\prod_i^n Y_k ^{-\frac{n}{2}}> \alpha^{-\frac{n}{2}})\\
& \leq \EE[\prod_i^n Y_k^{-\frac{1}{2}}] \alpha^\frac{n}{2}.
\end{align*}
Setting $\alpha = (eC')^{-2}$ we obtain
\begin{align*}
\Pro (\vol(S (0,X_1,\dots,X_n)) < \frac{\alpha^nL_K^n}{\sqrt{n!}}) \leq \frac{1}{2}e^{-n}.
\end{align*}
The result follows by applying Stirling formula.
\end{proof}

Based on the arguments given in the recent paper of Nasz{\'o}di \cite{naszodi2016proof} and with Propositions ~\ref{Propo 1} and~\ref{Propo volumen} at hand, we can now give a proof of Theorem~\ref{Theorem B}.

\begin{figure}

\begin{center}
\begin{tikzpicture}
\definecolor{cof}{RGB}{20,70,200}
\definecolor{pur}{RGB}{0,70,162}
\definecolor{greeo}{RGB}{91,173,69}
\definecolor{greet}{RGB}{100,0,0}
\pgfdeclarelayer{bg}
\pgfsetlayers{main,bg}

\begin{scope}
\node at (2,-4) {\large $K$};
\def \pa {(1,-3)};
\def \pb {(1,3)};
\def \pc {(4,0)};
\def \pd {(3,-3)};
\def \pe {(3,3)};


  \draw (0,0) arc (180:360:2cm and 0.6cm);
 \draw[dashed] (0,0) arc (180:0:2cm and 0.6cm);

  \draw \pa arc (180:360:1cm and 0.3cm);
 \draw[dashed] \pa arc (180:0:1cm and 0.3cm);

  \draw \pb arc (180:360:1cm and 0.3cm);
 \draw[dashed] \pb arc (180:0:1cm and 0.3cm);

\draw (0,0) -- \pa;
\draw (0,0) -- \pb;
\draw \pc -- \pd;
\draw \pc -- \pe;

\end{scope}

\begin{scope}[xshift = 1.5cm,yshift = -1cm]
\def \pa {(-0.7741,0.0764)};
\def \pb {(-0.571,-0.524)};
\def \pc {(0.971,-0.676)};
\def \pd {(0.370,1.124)};
\def \pw {(0,0)};
\def \pz {(0.778722777,2.35957407)};
\def \labS {(0.2,-1)}
\def \fiS {(-0.3,2)}


\node at \pw{$\bari(T)$};

\node at (0.77,2.6){$w$};

\node at \labS{$T$};
\node [thick] at \fiS{\textbf{$S=\varphi(T)$}};
\def \fia {(-0.065,1.1)};
\def \fib {(-0,0.75)};
\def \fic {(0.87,0.6553278)};
\def \fid {(0.6,1.879328)};

\draw [thick][top color=gray,opacity = 0.8] \fia -- \fib -- \fid;
\draw [thick][top color=gray,opacity = 0.8] \fic -- \fia -- \fid;
\draw [thick][top color=gray,opacity = 0.8] \fib -- \fic -- \fid;
\draw [thick][top color=gray,opacity = 0.8] \fib -- \fic -- \fia;

\draw \pa -- \pb -- \pd;
\draw \pc -- \pa -- \pd;
\draw \pb -- \pc -- \pd;
\draw \pb -- \pc -- \pa;

\draw [dashed] \pa -- \pz;
\draw [dashed] \pb -- \pz;
\draw [dashed] \pc -- \pz;

\end{scope}
\node at (1.87,0.25){$0$};
\end{tikzpicture}
\end{center}
\caption{Construction involved in the proof of Theorem~\ref{Theorem B}.}
\label{const}

\end{figure}

\begin{proof}[Proof of Theorem~\ref{Theorem B}]
Let $K \subset \R^n$ be an isotropic convex body and $X_1, \dots, X_n$ be independent random vectors uniformly distributed on $K$.
Denote by $T$ the simplex $S(0, X_1, \dots, X_n)$ and by $u$ its barycenter; i.e., $u= \frac{1}{n+1} \sum_{i=1}^n X_i$. By Proposition~\ref{Propo 1} there is an absolute constant $c_1>0$ such that
\begin{equation}
\Pro\left\{  \left\Vert u \right\Vert \leq c_1 L_K \right\} > 1 - \frac{1}{2} e^{-n}.
\end{equation}

On the other hand, by Proposition~\ref{Propo volumen}, we know that there is an absolute constant $c_2>0$ such that
\begin{equation}
\Pro\left\{  \vol\left(T \right) \geq \frac{c_2^n L_K^n}{n^{\frac{n}{2}}} \right\} > 1 - \frac{1}{2} e^{-n}.
\end{equation}

By a well-known result of Kannan, Lov{\'a}sz and Simonovits \cite[Theorem 4.1.]{kannan1995isoperimetric}  we have that
\begin{equation}
\sqrt{\frac{n+2}{n}} L_K B_2^n \subset K
\end{equation}
(note that for the authors the definition of an isotropic convex body is different, that is why the constant $L_K$ in the theorem is missing).
Therefore, the vector $w:=- \frac{1}{c_1} u$ belongs to $K$ with probability greater than  $1 - \frac{1}{2} e^{-n}$.

It is easy to check that if we apply the homothetic transformation  with center $w$ and ratio
$$\lambda = \frac{\Vert w \Vert}{\Vert w - u \Vert} = \frac{\Vert w \Vert}{\Vert w \Vert  + \Vert u \Vert} = \frac{1}{1+c_1} > 0$$
to the simplex $T$, we obtain another simplex $S$ with barycenter at the origin (see the Figure \ref{const}) such that

\begin{equation}
\vol(S) \geq \lambda^n \vol(T) \geq \lambda^n \cdot \frac{c_2^n L_K^n}{n^{\frac{n}{2}}}.
\end{equation}
Denote by $\bar{X}:= \frac{1}{n+1} \sum_{i=1}^n X_i$. Therefore, the function $f_n : \underbrace{\R^n \times \dots \times \R^n}_n \to \mathcal S_0^n$ we are looking for can be defined by
\begin{align*}
f_n(X_1,  \dots, X_n)  &:= \varphi(S(0, X_1, \dots, X_n))   \\
& = \frac{1}{1+c_1}S\left(-\bar{X}, X_1 - \bar{X},  \dots, X_n - \bar{X}\right).\\
\end{align*}
This concludes the proof.
\end{proof}

\subsection{Deduction of Theorem~\ref{main theorem} and its correct asymptotic behavior}\label{Section 3}

In this section we show how to deduce our main result, Theorem~\ref{main theorem} from Theorem~\ref{Theorem A}. We also show that the volume ratio, up to the absolute constants, is sharp.

We start with a well known lemma.
\begin{lemma}\label{Mahler}
For any simplex $S \subset \R^n$ with barycenter at the  origin, we have
\begin{equation} \label{Mahler simplices}
\M(S) = \frac{(n+1)^{(n+1)}}{{(n!)}^2}.
\end{equation}
\end{lemma}
Unfortunately, we could not find an exact reference of the previous lemma. We include a sketch of its proof: the Mahler product of a simplex is invariant under linear transformations, then it is possible to compute it using a particular example (all simplices belong to the same equivalence class). In particular, let $\{e_i\}_{i=1}^n$ be the canonical basis and consider the simplex $S:=S(e_1, e_2, \ldots, e_n, -\sum_{i=1}^n e_i),$ then its volume is $(n+1)/n!$. On the other hand, its polar $S^{\circ}$ is the simplex $S(v_0, v_1, \dots, v_n)$ where $v_0 = \sum_{i=1}^n e_i$ and $v_j = v_0 - (n+1) e_j$ for $1 \leq j \leq n$, whose volume is $(n+1)^n/n!$.

We can now give a proof of Theorem~\ref{main theorem}.

\begin{proof}[Proof of Theorem~\ref{main theorem} and Equation \eqref{boundsimplexconcteisotropianosim}]
Let $K \subset \R^n$ be an arbitrary convex set with barycenter at the origin.
By the Rogers-Shephard inequality \cite[Theorem 1.5.2]{artstein2015asymptotic} the centrally symmetric so-called difference body $D(K)=K-K$ contains $K$ and fulfills
\begin{equation}\label{eq0}
\left(\frac{\vol(D(K))}{\vol(K)} \right)^{1/n} \leq 4.
\end{equation}

By Equation~\eqref{desigualdad adentro} applied to the body $D(K)^{\circ}$ there is a simplex with barycenter at the origin $T \subset D(K)^{\circ}$ such that
\begin{equation} \label{eq a}
\left(\frac{\vol(D(K)^{\circ})}{\vol(T)}\right)^{1/n} \leq c \frac{\sqrt{n}}{L_{D(K)^{\circ}}},
\end{equation}
where $c>0$ is an absolute constant.

Consider $S$ the simplex $T^{\circ}$. It is not difficult to see that $S$ has also barycenter at the origin and obviously $S \supset D(K)$.
Now,

\begin{equation} \label{eq b}
 \frac{\vol(S)}{\vol(D(K))} =  \frac{\vol(S)\vol(T)}{\vol(D(K))\vol(D(K)^{\circ})}  \cdot \frac{\vol(D(K)^{\circ})}{\vol(T)}.
\end{equation}

By Lemma~\ref{Mahler}, the Bourgain-Milman inequality \cite[Theorem 8.2.2.]{artstein2015asymptotic} and  Stirling formula we have
\begin{equation}  \label{eq c}
 \left( \frac{\vol(S)\vol(T)}{\vol(D(K))\vol(D(K)^{\circ})} \right)^{1/n} \leq c
\end{equation}
for an absolute constant $c>0$.

The result now follows immediately form Equations  \eqref{eq a}, \eqref{eq b}, \eqref{eq c} and the fact that $D(K) \supset K$ and hence $S \supset K$.
\end{proof}

As we can see in the following example, the asymptotic behavior of the volume ratio given in Theorem~\ref{main theorem} cannot be improved.
\begin{figure}
\begin{center}

\begin{tikzpicture}
\definecolor{cof}{RGB}{20,70,200}
\definecolor{pur}{RGB}{0,70,162}
\definecolor{greeo}{RGB}{91,173,69}
\definecolor{greet}{RGB}{100,0,0}
\begin{scope}[xshift =5cm, yshift =5cm]
\def \pa {(2.0,-0.44721374)};
\def \pb {(-2.0,-0.44721374)};
\def \pc {(0.0,2.3445802)};
\def \pd {(0.0,-1.4501528)};
\draw  \pa -- \pc -- \pd ;
\draw \pb -- \pd -- \pc;
\draw [dashed]\pa -- \pb;
\draw \pb-- \pc;
\draw \pd -- \pa;
\draw \pb -- \pd -- \pa;
\end{scope}

\begin{scope} [xshift =5cm, yshift =5.1cm]
 \draw (-1,0) arc (180:360:1cm and 0.5cm);
 \draw[dashed] (-1,0) arc (180:0:1cm and 0.5cm);
 \draw (0,1) arc (90:270:0.5cm and 1cm);
 \draw[dashed] (0,1) arc (90:-90:0.5cm and 1cm);
 \draw (0,0) circle (1cm);
   \shade[ ball color=gray,opacity=0.50] (0,0) circle (1cm);
\end{scope}

    \end{tikzpicture}

\end{center}

    \caption{The simplex of minimal volume enclosing the Euclidean ball is the regular simplex circumscribing it.}
\label{simplex y bola}
\end{figure}

\begin{example}[The minimal volume simplex for the Euclidean ball] \label{ejemplo bola}
Let $K:=B_2^n$, the Euclidean ball, and $S \supset B_2^n$ the regular simplex circumscribing $K$.
As we can infer from the proof of \cite[Theorem 2.4.8. (ii)]{artstein2015asymptotic} $K=B_2^n$ is the maximal volume ellipsoid inside $S$ or, in other words, $S$ is in John position.

Let us see that $S$ is the minimal volume simplex containing $K$. If not, then there is a simplex $T \subset \R^n$ enclosing the ball with $\vol(T) < \vol(S)$.
Consider the linear transformation $A \in GL(n)$ such that $A(S)=T$; then, $\vert \det(A) \vert < 1$. Therefore $A^{-1}(B_2^n)$ is an ellipsoid with volume greater that $\vol(B_2^n)$ inside $S$, which is a contradiction.

If we compute the volumes (for the simplex it is easier to do it working in $\R^{n+1}$ on the hyperplane $\sum_{i} x_i =1$) we have:
\begin{align}
Vol(S) &= \frac{ (n+1)^{\frac{n+1}{2}}n^{\frac{n}{2}}}{n!}, \\
Vol(K) &= \frac{\pi^{\frac{n}{2}}}{\Gamma(\frac{n}{2} + 1)}.
\end{align}
Using Stirling formula we therefore get
\begin{equation}
\left(\frac{Vol(S)}{Vol(K)}\right)^{\frac{1}{n}} \approx \widetilde{d} \sqrt{n},
\end{equation}
for an absolute constant $\widetilde{d}>0$.
\end{example}

\section{The case of the cube and a non-probabilistic proof of Theorem~\ref{main theorem}}\label{graciasref}

As mentioned, for $n=2$ the cube has the largest volume ratio (respect to the simplex of minimal volume containing it); for $n=3$ the same is conjectured. One should expect that a similar phenomenon occurs in high dimensions but, as we can see in the following example, the volume ratio of the cube is uniformly bounded. Moreover, we show that the simplex can be taken with the same barycenter as the cube.

\begin{example} \label{ejemplocubo}
Let $K$ be the cube $[-\frac{1}{2},\frac{1}{2}]^n \subset \R^n$. There is a centered simplex $S$ such that $K \subset S$ and
\begin{equation}
\left(\frac{Vol(S)}{Vol(K)}\right)^{\frac{1}{n}} \leq c,
\end{equation}
for an absolute constant $c>0$.
\end{example}
\begin{proof}

Denote by $\mathbbm{1}$ the vector in $\R^n$ defined as $\sum_{j=1}^n e_j$. Consider the simplex
$$S:=S\left(-\frac{n}{2} \mathbbm{1}, n e_1 - \frac{1}{2} \mathbbm{1}, n e_2  -\frac{1}{2} \mathbbm{1}, \dots, n e_n  -\frac{1}{2} \mathbbm{1}\right).$$

It is easy to see that $\bari(S)=\bari(K)=0$.
Observe also that the cube $[-\frac{1}{2},\frac{1}{2}]^n$ is included in the simplex $T:=S\left(-\frac{1}{2} \mathbbm{1}, n e_1 - \frac{1}{2} \mathbbm{1}, n e_2  -\frac{1}{2} \mathbbm{1}, \dots, n e_n - \frac{1}{2} \mathbbm{1}\right)$. Indeed,
all the points that lie in the cube have coordinates greater than or equal to $-\frac{1}{2}$ and their sum is less than or equal to $\frac{n}{2}$.
It remains to see that the point $-\frac{1}{2} \mathbbm{1}$ belongs to $S$, but $-\frac{1}{2} \mathbbm{1}$  is exactly $t (-\frac{n}{2}) \mathbbm{1} + (1-t) \frac{1}{2} \mathbbm{1}$ for $t= \frac{2}{n+1}$.
An easy computation proves that the volume of $S$ is exactly $ \frac{n^n (n+1)}{2 n!}$.
\end{proof}

\begin{figure}[H]
\begin{center}
\begin{tikzpicture}
\definecolor{greet}{RGB}{50,0,50}
\begin{scope}[yshift =- 14cm]
\begin{scope}
\def \pa {(-1.52,-0.83)};
\def \pb {(1.74,0.42)};
\def \pc {(0.42,2.86)};
\def \pd {(2.41,-0.8)};
\def \po {(0,0)};
\def \pm {(-1.9,-0.83)}
\node at \pm{$-\frac{3}{2}\mathbbm {1}$};

\draw[dotted] \po -- \pa;
\draw \pb -- \pc -- \pa;
\draw  \pc -- \pd -- \pb -- \pc;
\draw \pd -- \pb; 
\draw \pa -- \pd;
\draw \pb -- \pc -- \pa;
\draw \pb -- \pc -- \pa;
\draw [dashed] \pa -- \pb;
\fill [color=greet,opacity=0.04] \pa --\pb -- \pd;
\end{scope}
\begin{scope}
\def \pa {(0,0)};
\def \pb {(0.8,-0.27)};
\def \pc {(0.14,0.95)};
\def \pd {(0.94,0.69)};
\def \pe {(0.58,0.14)};
\def \pf {(1.38,-0.13)};
\def \pg {(0.72,1.09)};
\def \ph {(1.52,0.83)};
\def \pn {(-0.4,0)};
\def \pp{(1.92,0.83)}
\node at \pn{$-\frac{1}{2}\mathbbm {1}$};
\node at \pp{$\frac{1}{2}\mathbbm {1}$};
(0.0,0.0),(0.8,-0.27),(0.14,0.95),(0.94,0.69),(0.58,0.14),(1.38,-0.13),(0.72,1.09),(1.52,0.83)


\draw  \pa -- \pb -- \pd ;
\draw \pa -- \pc  -- \pd;
\draw \pa -- \pe -- \pf;
\draw \pb -- \pf -- \ph;
\draw \pc -- \pg -- \ph;
\draw \pg -- \pe;
\draw \pd -- \ph;
\draw [top color =gray,opacity = 0.8] \pg -- \ph -- \pd -- \pc -- \pg;
\draw [top color=gray,opacity = 0.8]\pc -- \pa -- \pb -- \pd;
\draw  [top color=gray,opacity = 0.8]\pd -- \ph -- \pf -- \pb;
\draw [top color=gray,opacity = 0.8]\pa -- \pe -- \pf -- \pb;
\draw [top color=gray,opacity = 0.8]\pg -- \ph -- \pf -- \pe;
\draw  [top color=gray,opacity = 0.8]\pg -- \pe -- \pa -- \pc;
\draw[dotted] \pa--\ph;
\end{scope}
(-1.52,-0.83),(1.74,0.42),(0.42,2.86),(2.41,-0.8)
(1.2,0.6)

\end{scope}
\end{tikzpicture}

\end{center}
\caption{The simplex $S\left(-\frac{3}{2} \mathbbm{1}, 3 e_1 - \frac{1}{2} \mathbbm{1}, 3 e_2  -\frac{1}{2} \mathbbm{1}, 3 e_3  -\frac{1}{2} \mathbbm{1}\right)$ enclosing the cube $[-\frac{1}{2},\frac{1}{2}]^3 \subset \R^3$ as in Example~\ref{ejemplocubo}.}
\end{figure}

We end the article giving a non-probabilistic proof of Theorem~\ref{main theorem}. This relies on the Rogers-Shephard inequality, the Dvoretzky-Rogers theorem and previous estimate for the cube.

\begin{proof}[Proof of Theorem~\ref{main theorem}]
Again, by the the Rogers-Shephard inequality \cite[Theorem 1.5.2]{artstein2015asymptotic} (see \eqref{eq0}) we can suppose, without loss of generality that $K$ is centrally symmetric. Using a well-known result of Dvoretzky and Rogers \cite[Theorem 5A]{dvoretzky1950absolute} (see also \cite{pelczynski1991parallelepipeds}) there is a centrally symmetric parallelepiped $ P \supset K$ such that
\begin{equation}\label{eq1}
\left(\frac{\vol(P)}{\vol(K)} \right)^{1/n} \leq c \sqrt{n},
\end{equation}
for some absolute constant $c>0$.
The result now follows combining Equation \eqref{eq1} and the bound given in Example \ref{ejemplocubo} for the simplex containing the parallelepiped $P$ (with, of course, \eqref{invariantetransf}).
\end{proof}

Comparing the result obtained with this technique with \eqref{boundsimplexconcteisotropia} and \eqref{boundsimplexconcteisotropianosim}, one should note that the isotropic constant is missing (maybe in case the isotropic constant conjecture \cite[Conjecture 10.1.7]{artstein2015asymptotic} is false,  \eqref{boundsimplexconcteisotropia} or \eqref{boundsimplexconcteisotropianosim} could give better estimates for certain bodies).

\subsection{Acknowledgement}
The authors are grateful to the anonymous referee for the clever insight regarding Problem~\ref{problem0} which gave origin to the previous section.

\end{document}